\documentclass[11pt]{article}
\usepackage{amsmath, amsthm, amscd, amsfonts, amssymb, graphicx, enumerate, color, hyperref}

\theoremstyle{definition}
\newtheorem{theorem}{Theorem}[section]

\newtheorem{lemma}[theorem]{Lemma}
\newtheorem{corollary}[theorem]{Corollary}
\newtheorem{proposition}[theorem]{Proposition}
\newtheorem{remark}[theorem]{Remark}

\begin{document}
\title{\bf Improvements of the weighted Hermite-Hadamard inequality and applications to mean inequality}
\author{Shigeru Furuichi$^1$\footnote{E-mail:furuichi@chs.nihon-u.ac.jp}, Nicu\c{s}or Minculete$^2$\footnote{E-mail:minculeten@yahoo.com} and Hamid Reza Moradi$^3$\footnote{E-mail:hrmoradi.68@gmail.com}\\
$^1${\scriptsize	Department of Information Science, College of Humanities and Sciences, Nihon University, Setagaya-ku, Tokyo,
Japan}\\
$^2${\scriptsize Faculty of Mathematics and Computer Science, Transilvania University, Iuliu Maniu street 50, 500091 Bra\c{s}ov, Romania}\\
$^3${\scriptsize Department of Mathematics, Mashhad Branch, Islamic Azad University, Mashhad, Iran}}
\date{}
\maketitle
{\bf Abstract.} 
This paper aims to characterize the function appearing in the weighted Hermite-Hadamard inequality. We provide improved inequalities for the weighted means as applications of the obtained results. Modifications of the weighted Hermite-Hadamard inequality are also presented. Our results contain some exciting inequalities and extensions of the known results. 
\vspace{3mm}

\vspace{3mm}

\vspace{3mm}


\section{Introduction and Preliminiaries}
For $a,b>0$ and $0\le\nu\le1$, the weighted arithmetic-geometric mean inequality asserts that $a\sharp_\nu b \le a\nabla_\nu b$,  where $a \sharp_\nu b:= a^{1-\nu}b^\nu $ and $a\nabla_\nu b:=(1-\nu)a+\nu b$ are named the weighted geometric mean and the weighted arithmetic mean, respectively. We use the symbols $\nabla$ and $\sharp$ instead of ${{\nabla }_{{1}/{2}\;}}$ and ${{\sharp }_{{1}/{2}\;}}$.
During the past decades, the study of inequalities involving mathematical means has attracted many mathematicians; see, for example, \cite{DraAg, F_Y_M, F_M, F_M_2, FM2020, fms2021}.

Recently, in  \cite[Theorem 2.2]{PSMA2016}, the weighted logarithmic mean was introduced in the following structure:
\begin{equation}\label{8}
L_\nu(a,b) := \frac{1}{\log a-\log b}\left\{\frac{1-\nu}{\nu}(a-a^{1-\nu}b^\nu)+\frac{\nu}{1-\nu}(a^{1-\nu}b^\nu-b)\right\}
\end{equation}
for $a,b >0$, $a\neq b$ with $\nu \in (0,1)$ and $L_\nu(a,a)=a$. For $\nu={1}/{2}\;$, \eqref{8} reduces to the logarithmic mean $L_{1/2}(a,b)=L(a,b):= \dfrac{a-b}{\log a-\log b}$. Besides, it has been shown that
\begin{equation}\label{sec1_eq01}
a\sharp_\nu b \leq L_\nu(a,b) \leq a\nabla_\nu b.
\end{equation}
Inequality \eqref{sec1_eq01} provides a modification of the famous Young's inequality 
\begin{equation*}
ab\leq \frac{1}{p\log a-q\log b}\left(\frac{q}{p}(a^p-ab)+\frac{p}{q}(ab-b^q)\right)\leq\frac{a^p}{p}+\frac{b^q}{q}
\end{equation*}
for $a,b >0$, $a^p\neq b^q$ with $p,q>1$ and $1/p+1/q=1$.

Notice that inequality \eqref{sec1_eq01} is an immediate consequence of the following generalization of the Hermite-Hadamard inequality (see \cite[Theorem 2.1]{PSMA2016})
\begin{equation}\label{6}
\begin{aligned}
  & f\left( a{{\nabla }_{\nu}}b \right) \\ 
 & \le \left( 1-\nu \right)\int\limits_{0}^{1}{f\left( \nu \lambda (b-a)+a \right)d\lambda }+\nu\int\limits_{0}^{1}{f\left( (1-\nu)\lambda (b-a)+\nu b+(1-\nu)a \right)d\lambda } \\ 
 & \le f\left( a \right){{\nabla }_{\nu}}f\left( b \right) \\ 
\end{aligned}
\end{equation}
for a convex Riemann integrable function $f:[a,b]\to\mathbb{R}$ and $a,b > 0$ with $\nu\in [0,1]$. Indeed, by letting $\nu=1/2$ in \eqref{6}, we recover the Hermite-Hadamard inequality:
\begin{equation}\label{7}
f\left( a\nabla b \right)\le \int\limits_{0}^{1}{f\left( a{{\nabla }_{\lambda}}b \right)d\lambda}\le f\left( a \right)\nabla f\left( b \right).
\end{equation}
For additional refinements and applications related to Hermite-Hadamard inequality, see \cite{FaLaBa, moradi1, moradi2}. 

Since $a{{\nabla }_{0}}b=a$, $a{{\nabla }_{1}}b=b$, $a{{\nabla }_{1-t}}b=b{{\nabla }_{t}}a$, and $\left( a{{\nabla }_{\alpha }}b \right){{\nabla }_{\gamma }}\left( a{{\nabla }_{\beta }}b \right)=a{{\nabla }_{\left( 1-\gamma  \right)\alpha+\gamma \beta }}b$ with $\alpha ,\beta ,\gamma \in \left[ 0,1 \right]$, inequality \eqref{6} can be written as
\begin{equation}\label{Pal_gen_HHI}
f\left( a{{\nabla }_{\nu}}b \right)\le {{\mathfrak C }_{f,\nu}}\left( a,b \right)\le f\left( a \right){{\nabla }_{\nu}}f\left( b \right)
\end{equation}
where
\begin{equation}\label{sec1_eq02.1.0}
{{\mathfrak C }_{f,\nu}}\left( a,b \right)=\left( \int\limits_{0}^{1}{f\left( a{{\nabla }_{\nu \lambda }}b \right)d\lambda } \right){{\nabla }_{\nu}}\left( \int\limits_{0}^{1}{f\left( b{{\nabla }_{\left( 1-\nu \right)\lambda }}a \right)d\lambda } \right),
\end{equation}
due to
$$
\int\limits_{0}^{1}{f\left( b{{\nabla }_{\left( 1-\nu \right)(1-\lambda )}}a \right)d\lambda } =
 \int\limits_{0}^{1}{f\left( b{{\nabla }_{\left( 1-\nu \right) \mu}}a \right)d\mu }. 
$$

In \cite{PSMA2016},  the representing function of the weighted logarithmic mean, i.e.,
\begin{equation}\label{rep_func_log_mean}
f_\nu(t):=\dfrac{1}{\log t}\left\{\dfrac{1-\nu}{\nu}(t^\nu-1)+\dfrac{\nu}{1-\nu}(t-t^\nu)\right\}=L_\nu(1,t),
\quad (1\neq t>0)
\end{equation}
was studied and characterized by the following inequalities:
 $$t^\nu\leq f_\nu(t)\leq \frac{1}{2}\left(t^\nu+(1-\nu)+\nu t\right)\leq (1-\nu)+\nu t.$$

The following results have been established in \cite{F_M}:
\begin{theorem}\label{sec1_theorem01}
Let $f:[a,b]\to \mathbb{R}$ be a convex function. Then for any $\nu\in[0,1]$, 
\begin{equation*}
f\left(a\nabla_\nu b\right) \leq \mathfrak R^{(1)}_{f,\nu} (a,b) \leq \mathfrak C _{f,\nu}(a,b) \leq \mathfrak R^{(2)}_{f,\nu} (a,b) \leq f(a)\nabla_\nu f(b),
\end{equation*}
where
\begin{equation*}
\mathfrak R^{(1)}_{f,\nu} (a,b) := f(a\nabla_{\frac{\nu}{2}}b)\nabla_\nu f(a\nabla_{\frac{1+\nu}{2}}b),
\end{equation*}
and
\begin{equation*}
\mathfrak R^{(2)}_{f,\nu} (a,b) := \left(f(a)\nabla_\nu f(b)\right)\nabla\left( f(a \nabla_\nu b) \right).
\end{equation*}
\end{theorem}
\begin{corollary}\label{sec1_cor01}
Let $a,b>0$ and $\nu\in (0,1)$. Then
\begin{equation*}
a\sharp_\nu b \leq \left(a\sharp_{\frac{\nu}{2}} b\right)\nabla_\nu \left(a \sharp_{\frac{1+\nu}{2}}b\right)\leq L_\nu(a,b)\leq \left(a\nabla_\nu b\right) \nabla\left(a\sharp_\nu b\right) \leq a \nabla_\nu b.
\end{equation*}
\end{corollary}

 In this paper, we refine inequalities \eqref{sec1_eq01}. Refinement of the Hermite-Hadamard inequality is also provided. The inequalities demonstrated in the next section can be extended to the positive Hilbert space operators by utilizing the standard functional calculus. We leave this idea for the interested reader.

\section{Main Results}
We begin with the following lemma, which includes two identities for $\mathfrak C _{f,\nu}$.
\begin{lemma}
Let $f:[0,1)\cup(1,\infty)\to \mathbb{R}$ be a convex function. Then, for any $\nu\in(0,1)$, 
\begin{equation}\label{sec2_eq01}
 \mathfrak C _{f,\nu}(t,1) =
    \frac{1}{1-t}\left(\frac{1-\nu}{\nu}\int\limits_t^{(1-\nu)t+\nu}f(\lambda )d\lambda +\frac{\nu}{1-\nu}\int\limits_{(1-\nu)t+\nu}^1f(\lambda )d\lambda \right),
\end{equation}
and
\begin{equation}\label{sec2_eq01_002}
 \mathfrak C _{f,\nu}(1,t) =
    \frac{1}{t-1}\left(\frac{1-\nu}{\nu}\int\limits_1^{(1-\nu)+\nu t}f(\lambda )d\lambda +\frac{\nu}{1-\nu}\int\limits_{(1-\nu)+\nu t}^tf(\lambda )d\lambda \right).
\end{equation}
\end{lemma}
\begin{proof}
Putting $a=t$ and $b=1$ in  \eqref{sec1_eq02.1.0}, we deduce the equality \eqref{sec2_eq01}, with calculations. Relation \eqref{sec2_eq01_002} can be obtained likewise.
\end{proof}
\begin{remark}\label{1}
\hfill
\begin{itemize}
\item[(i)] If we take $f(\lambda )=\lambda $ in \eqref{sec2_eq01_002}, then we have $\mathfrak C _{\lambda,\nu }(1,t)=(1-\nu )+\nu t$ which is the representing function of the weighted arithmetic mean.
\item[(ii)] If we take $\nu ={1}/{2}\;$ in  \eqref{sec2_eq01} and \eqref{sec2_eq01_002}, then we reach
$$
    \mathfrak C _{f,\frac{1}{2}}(t,1) = \frac{1}{1-t} \int\limits_t^1f(\lambda )d\lambda =\frac{1}{t-1}\int\limits_1^tf(\lambda )=\mathfrak C _{f,\frac{1}{2}}(1,t),\quad (1\neq t >0).
$$
\item[(iii)] For $t=0$, in equality \eqref{sec2_eq01}, we obtain
\begin{equation}\label{remark2.2_eq01}
    \mathfrak C _{f,\nu }(0,1) = 
    \frac{1-\nu }{\nu }\int\limits_0^{\nu }f(\lambda )d\lambda +\frac{\nu }{1-\nu }\int\limits_{\nu }^1f(\lambda )d\lambda .
\end{equation}
If we take $f(\lambda )=t^\lambda $ in \eqref{remark2.2_eq01}, then we deduce  
$\mathfrak C _{t^\lambda ,\nu }(0,1)=f_\nu (t)$, where $f_\nu (t)$ is given as in \eqref{rep_func_log_mean}.
\end{itemize}
\end{remark}

On account of Remark \ref{1}, it is interesting to study the function $\mathfrak C _{f,\nu }(t,1)$. The following result presents an upper and a lower bound for $\mathfrak C _{f,\nu }(t,1)$.
\begin{theorem}\label{sec2_lemma01}
Let $f:[0,1)\cup(1,\infty)\to \mathbb{R}_+$ be a convex function. Then for any $\nu \in(0,1)$, 
\begin{equation}\label{sec2_eq13}
    \min\left\{\frac{1-\nu }{\nu },\frac{\nu }{1-\nu }\right\}\mathfrak C _{f,\frac{1}{2}}(t,1)
        \leq 
    \mathfrak C _{f,\nu }(t,1)
        \leq 
    \max\left\{\frac{1-\nu }{\nu },\frac{\nu }{1-\nu }\right\}\mathfrak C _{f,\frac{1}{2}}(t,1).
\end{equation}
\end{theorem}
\begin{proof}
Employing Remark \ref{1} (ii), we have
{\small
\[\begin{aligned}
   \min \left\{ \frac{1-\nu }{\nu },\frac{\nu }{1-\nu } \right\}{{\mathfrak C }_{f,\frac{1}{2}}}\left( t,1 \right)&=\min \left\{ \frac{1-\nu }{\nu },\frac{\nu }{1-\nu } \right\}\frac{1}{1-t}\int\limits_{t}^{1}{f\left( \lambda  \right)d\lambda } \\ 
 & =\min \left\{ \frac{1-\nu }{\nu },\frac{\nu }{1-\nu } \right\}\frac{1}{1-t}\left( \int\limits_{t}^{\left( 1-\nu  \right)t+\nu }{f\left( \lambda  \right)d\lambda }+\int\limits_{\left( 1-\nu  \right)t+\nu }^{1}{f\left( \lambda  \right)d\lambda } \right) \\ 
 & \le \frac{1}{1-t}\left( \frac{1-\nu }{\nu }\int\limits_{t}^{\left( 1-\nu  \right)t+\nu }{f\left( \lambda  \right)d\lambda }+\frac{\nu }{1-\nu }\int\limits_{\left( 1-\nu  \right)t+\nu }^{1}{f\left( \lambda  \right)d\lambda } \right) \\ 
 & \le \max \left\{ \frac{1-\nu }{\nu },\frac{\nu }{1-\nu } \right\}\frac{1}{1-t}\left( \int\limits_{t}^{\left( 1-\nu  \right)t+\nu }{f\left( \lambda  \right)d\lambda }+\int\limits_{\left( 1-\nu  \right)t+\nu }^{1}{f\left( \lambda  \right)d\lambda } \right) \\ 
 & =\max \left\{ \frac{1-\nu }{\nu },\frac{\nu }{1-\nu } \right\}\frac{1}{1-t}\int\limits_{t}^{1}{f\left( \lambda  \right)d\lambda } \\ 
 & =\max \left\{ \frac{1-\nu }{\nu },\frac{\nu }{1-\nu } \right\}{{\mathfrak C }_{f,\nu }}\left( t,1 \right).  
\end{aligned}\]
}
Consequently, we prove the inequality of the statement.
\end{proof}

\begin{remark}
Letting $t=0$ in \eqref{sec2_eq13}. Then for any $\nu \in(0,1)$,
\begin{equation}\label{sec2_eq14}
    \min\left\{ \frac{1-\nu }{\nu },\frac{\nu }{1-\nu }\right\}\int\limits_0^1f(\lambda )d\lambda 
        \leq 
    \mathfrak C _{f,\nu }(0,1)\leq \max\left\{ \frac{1-\nu }{\nu },\frac{\nu }{1-\nu }\right\}\int\limits_0^1f(\lambda )d\lambda .
\end{equation}
If we take $f(\lambda )=t^\lambda $ in \eqref{sec2_eq14}, we infer
\begin{equation*}
     \min\left\{ \frac{1-\nu }{\nu },\frac{\nu }{1-\nu }\right\}\underbrace{\frac{t-1}{\log t}}_{L_\frac{1}{2}(t,1)}
        \leq 
     f_\nu (t)
        \leq 
     \max\left\{ \frac{1-\nu }{\nu },\frac{\nu }{1-\nu }\right\}\underbrace{\frac{t-1}{\log t}}_{L_\frac{1}{2}(t,1)}.
\end{equation*}
The above inequalities have been demonstrated in \cite[Theorem 2.2]{F_Y_M}. More precisely, Theorem \ref{sec2_lemma01} provides an extension of  \cite[Theorem 2.2]{F_Y_M}.
\end{remark}

In the following lemma, the difference between the weighted arithmetic mean and the weighted geometric mean has been represented by the representing function of the weighted logarithmic mean $L_{\nu }(t,1)$. 
\begin{lemma}\label{sec2_lemma01.2}
Let $t\in[0,1)\cup(1,\infty)$ and $\nu \in(0,1)$. Then
\begin{equation*}
    L_\frac{1}{2}(t,1)-L_\frac{1}{2}(t^\nu ,1) =\frac{(1-\nu )+\nu t-t^\nu }{\nu  \log t}.
\end{equation*}
\end{lemma}
\begin{proof}
It is easy to see that $$\nu \log t\left\{L_\frac{1}{2}(t,1)-L_\frac{1}{2}(t^\nu ,1)\right\}=(\nu \log t) L_\frac{1}{2}(t,1)-(\nu \log t)L_\frac{1}{2}(t^\nu ,1)=\nu t-\nu -t^\nu +1,$$
which proves the equality of the statement.
\end{proof}

In the sequel, we need the following refinements and reverses of Young inequality.
\begin{enumerate}
    \item[(i)] Kittaneh-Manasrah's inequality \cite{m3, m4}: For any $t>0$,
        \begin{equation}\label{sec2_eq17}
            r(\sqrt{t}-1)^2 \leq (1-\nu )+\nu t-t^\nu  \leq R(\sqrt{t}-1)^2,
        \end{equation}
    where $r=\min\{\nu ,1-\nu \}$, $R=\max\{\nu ,1-\nu \}$, and $\nu \in[0,1]$.
  
    \item[(ii)] Cartwright-Field's inequality \cite{m2}: For any $t>0$ and $0\le \nu \le 1$,
        \begin{equation}\label{sec2_eq18}
            \frac{1}{2}\nu (1-\nu )\frac{(t-1)^2}{\max\{t,1\}}
                \leq 
            (1-\nu )+\nu t-t^\nu 
                \leq 
            \frac{1}{2}\nu (1-\nu )\frac{(t-1)^2}{\min\{t,1\}}.
        \end{equation}

    \item[(iii)] Alzer-Fonseca-Kova\v cec's inequality \cite{m1}: For any $t>0$ and $0<\nu ,\lambda <1$,
        \begin{equation}\label{sec2_eq19}
            \frac{1}{2}\nu (1-\nu )\min\{t,1\}\log^2t
                \leq 
            (1-\nu )+\nu t-t^\nu 
                \leq 
            \frac{1}{2}\nu (1-\nu )\max\{t,1\}\log^2t,
        \end{equation}
    and
        \begin{equation}\label{sec2_eq20}
        \begin{aligned}
            \min\left\{\frac{\nu }{\lambda },\frac{1-\nu }{1-\lambda }\right\}\left(\lambda  t+(1-\lambda )-t^\lambda \right)
                &\leq 
            (1-\nu )+\nu t-t^\nu 
                \\&\leq
            \max\left\{\frac{\nu }{\lambda },\frac{1-\nu }{1-\lambda }\right\}\left(\lambda  t+(1-\lambda )-t^\lambda \right).
        \end{aligned}
        \end{equation}
     In particular, if $\lambda =1-\nu $, in \eqref{sec2_eq20}, then 
     \begin{equation}\label{sec2_eq22}
            \begin{aligned}
               \min\left\{ \frac{1-\nu }{\nu },\frac{\nu }{1-\nu }\right\}\left((1-\nu )t+\nu -t^{1-\nu }\right)
                    &\leq 
                (1-\nu )+\nu t-t^\nu 
                    \\&\leq
                \max\left\{\frac{\nu }{1-\nu },\frac{1-\nu }{\nu }\right\}\left((1-\nu ) t+\nu -t^{1-\nu }\right).
            \end{aligned}
            \end{equation}
\end{enumerate}

Employing the above inequalities together with Lemma \ref{sec2_lemma01.2}, we get the following result:
       
 \begin{proposition}\label{sec2_theorem01.2}
Let $\nu \in(0,1)$. If $t>1$, then
\begin{equation*}
    \frac{r}{\nu }\frac{(\sqrt{t}-1)^2}{\log t}
        \leq 
    L_\frac{1}{2}(t,1)-L_\frac{1}{2}(t^\nu ,1)
        \leq 
    \frac{R}{\nu }\frac{(\sqrt{t}-1)^2}{\log t},
\end{equation*}
\begin{equation*}
    \frac{1-\nu }{2}\frac{(t-1)^2}{\max\{t,1\}\log t}
        \leq 
    L_\frac{1}{2}(t,1)-L_\frac{1}{2}(t^\nu ,1)
        \leq 
    \frac{1-\nu }{2}\frac{(t-1)^2}{\min\{t,1\}\log t},
\end{equation*}

\begin{equation*}
    \frac{1-\nu }{2}\min\{t,1\}\log t
        \leq 
    L_\frac{1}{2}(t,1)-L_\frac{1}{2}(t^\nu ,1)
        \leq 
    \frac{1-\nu }{2}\max\{t,1\}\log t,
\end{equation*}
and
\begin{equation*}
    \begin{aligned}
        &\frac{1-\nu }{\nu }\min\left\{\frac{\nu }{1-\nu },\frac{1-\nu }{\nu }\right\}\left(L_\frac{1}{2}(t,1)-L_\frac{1}{2}(t^{1-\nu },1)\right)
            \\&\leq 
        L_\frac{1}{2}(t,1)-L_\frac{1}{2}(t^\nu ,1)
            \\&\leq
        \frac{1-\nu }{\nu }\max\left\{\frac{\nu }{1-\nu },\frac{1-\nu }{\nu }\right\}\left(L_\frac{1}{2}(t,1)-L_\frac{1}{2}(t^{1-\nu },1)\right).
    \end{aligned}
\end{equation*}
The reversed inequalities hold when $0<t<1$.
\end{proposition}
\begin{proof}
Using Lemma \ref{sec2_lemma01.2}, we find 
$$(1-\nu )+\nu t-t^\nu  = \nu \log t\left(L_\frac{1}{2}(t,1)-L_\frac{1}{2}(t^\nu ,1)\right).$$
Replacing the expression $(1-\nu )+\nu t-t^\nu $ in inequalities \eqref{sec2_eq17}, \eqref{sec2_eq18}, \eqref{sec2_eq19}, and \eqref{sec2_eq22}, we deduce the inequalities from the statement.
\end{proof} 

We can obtain the alternative expression of the difference between the weighted arithmetic mean and the weighted geometric mean by the weighted logarithmic mean and the logarithmic mean. Related to this, we state the following lemma.
\begin{lemma}\label{sec2_lemma01.3}
Let $t\in[0,1)\cup(1,\infty)$ and $\nu \in(0,1)$. Then
\begin{equation*}
    L_\nu (t,1)-L_\frac{1}{2}(t,1)=
    \frac{(2\nu -1)}{\nu (1-\nu )\log t}\left\{(1-\nu )+\nu t-t^\nu \right\}.
\end{equation*}
\end{lemma}
\begin{proof}
Making the difference between the weighted logarithmic mean and the logarithmic mean of $t$ and $1$, we have: 
\begin{equation*}
    \begin{aligned}
    L_\nu (t,1)-L_\frac{1}{2}(t,1) 
    &= 
    \frac{1}{\log t}\left(\frac{1-\nu }{\nu }(t^\nu -1)+\frac{\nu }{1-\nu }(t-t^\nu )-t+1)\right)
        \\&=
    \frac{1}{\log t}\left\{\left(\frac{1-\nu }{\nu }-\frac{\nu }{1-\nu }\right)t^\nu +\left(\frac{\nu }{1-\nu }-1\right)t+1-\frac{1-\nu }{\nu }\right\}
        \\&=
    \frac{1-2\nu }{\log t}\left\{\frac{1}{\nu (1-\nu )}t^\nu -\frac{1}{1-\nu }t-\frac{1}{\nu }\right\}
        \\&=
    \frac{1-2\nu }{\nu (1-\nu )\log t}\left\{t^\nu -\nu t-(1-\nu )\right\}
    \end{aligned}
\end{equation*}
for all $t>0$, $t\neq 1$ and $\nu \in(0,1)$. 
\end{proof}

\begin{remark}
Using Lemma \ref{sec2_lemma01.3}, we can obtain similar results like Proposition \ref{sec2_theorem01.2} with the help of inequalities \eqref{sec2_eq17}, \eqref{sec2_eq18}, \eqref{sec2_eq19}, and \eqref{sec2_eq22}. However, we leave them for interested readers.
\end{remark}

Inequality \eqref{sec1_eq01} can be improved by using Theorem \ref{sec1_theorem01}. Indeed, we have:
\begin{theorem}\label{theorem_2.8.}
Let $a,b>0$, $a\neq b$. Then for any $\nu \in(0,1)$, 
\[\begin{aligned}
   a{{\sharp}_{\nu }}b&\le \left( a{{\sharp}_{\frac{3\nu }{4}}}b \right){{\nabla }_{\nu }}\left( a{{\sharp}_{\frac{1+3\nu }{4}}}b \right) \\ 
 & \le \left( \sqrt{a}{{\sharp}_{\nu }}\sqrt{b} \right){{L}_{\nu }}\left( \sqrt{a},\sqrt{b} \right) \\ 
 & \le \left( a{{\sharp}_{\nu }}b \right)\nabla \left( \left( a{{\sharp}_{\frac{\nu }{2}}}b \right){{\nabla }_{\nu }}\left( a{{\sharp}_{\frac{1+\nu }{2}}}b \right) \right) \\ 
 & \le \left( a{{\sharp}_{\frac{\nu }{2}}}b \right){{\nabla }_{\nu }}\left( a{{\sharp}_{\frac{1+\nu }{2}}}b \right) \\ 
 & \le {{L}_{\nu }}\left( a,b \right) \\ 
 & \le \left( a{{\sharp}_{\nu }}b \right){{\nabla }_{\nu }}\left( a{{\nabla }_{\nu }}b \right) \\ 
 & \le a{{\nabla }_{\nu }}b.  
\end{aligned}\]
\end{theorem}
\begin{proof}
We set $a=0$, $b=1$, and $f(\lambda )=t^\lambda ,\,\,(t>0)$ in Theorem \ref{sec1_theorem01}. Then we have
$$
    t^\nu 
    \leq \mathfrak R_{t^\lambda ,\nu }^{(1)}(0,1)
    \leq \mathfrak C _{t^\lambda ,\nu }(0,1)
    \leq \mathfrak R_{t^\lambda ,\nu }^{(2)}(t^\lambda ,1)
    \leq \nu t+(1-\nu ),
$$
where
$$ 
    \mathfrak R_{t^\lambda ,\nu }^{(1)}(0,1)=t^\frac{\nu }{2}\nabla_\nu  t^\frac{1+\nu }{2}=(1-\nu )t^\frac{\nu }{2}+\nu t^\frac{1+\nu }{2},
$$
and
$$
    \mathfrak R_{t^\lambda ,\nu }^{(2)}(0,1)=(\nu t+(1-\nu ))\nabla t^\nu =\frac{1}{2}[t^\nu +(1-\nu )+\nu t].
$$
That is, we obtain
\begin{equation}\label{sec2_eq11}
    t^\nu 
    \leq (1-\nu )t^\frac{\nu }{2}+\nu t^\frac{1+\nu }{2}
    \leq f_\nu (t)
    \leq \frac{1}{2}\left(t^\nu +(1-\nu )+\nu t\right)
    \leq \nu t+(1-\nu )
\end{equation}
for all $t\in[0,1)\cup(1,\infty)$ and $\nu \in(0,1)$.

If we replace $t$ by $t^\frac{1}{2}$ in inequality \eqref{sec2_eq11}, then we deduce the following sequence of inequalities:
$$
    t^\frac{\nu }{2}
    \leq(1-\nu )t^\frac{\nu }{4}+\nu t^\frac{1+\nu }{4}
    \leq f_\nu (t^\frac{1}{2})
    \leq \frac{1}{2}\left(t^\frac{\nu }{2}+(1-\nu )+\nu t^\frac{1}{2}\right)
    \leq \nu t^\frac{1}{2}+(1-\nu ).
$$
Multiplying by $t^\frac{\nu }{2}$ the above sequence of inequalities, we have
\begin{equation*}
    t^\nu 
    \leq (1-\nu )t^\frac{3\nu }{4}+\nu t^\frac{1+3\nu }{4}
    \leq t^\frac{\nu }{2}f_\nu (t^\frac{1}{2})
    \leq \frac{1}{2}\left(t^\nu +(1-\nu )t^\frac{\nu }{2}+\nu t^\frac{\nu +1}{2}\right)
    \leq \nu t^\frac{\nu +1}{2}+(1-\nu )t^\frac{\nu }{2},
\end{equation*}
for all $t\in (0,1)\cup(1,\infty)$ and $\nu \in (0,1)$.
From the first and the second inequalities in \eqref{sec2_eq11}, we find 
\begin{equation*}
\frac{1}{2}\left(t^\nu +(1-\nu )t^{\frac{\nu }{2}}+\nu t^{\frac{1+\nu }{2}}\right)\le (1-\nu )t^{\frac{\nu }{2}}+\nu t^{\frac{1+\nu }{2}}\le f_\nu (t).
\end{equation*}
Thus we have the inequalities
\begin{equation}\label{sec2_eq12_02}
\begin{aligned}
   {{t}^{\nu }}&\le \left( 1-\nu  \right){{t}^{\frac{3\nu }{4}}}+\nu {{t}^{\frac{1+3\nu }{4}}} \\ 
 & \le {{t}^{\frac{\nu }{2}}}{{f}_{\nu }}\left( {{t}^{\frac{1}{2}}} \right) \\ 
 & \le \frac{1}{2}\left( {{t}^{\nu }}+\left( 1-\nu  \right){{t}^{\frac{\nu }{2}}}+\nu {{t}^{\frac{1+\nu }{2}}} \right) \\ 
 & \le \left( 1-\nu  \right){{t}^{\frac{\nu }{2}}}+\nu {{t}^{\frac{1+\nu }{2}}} \\ 
 & \le {{f}_{\nu }}\left( t \right) \\ 
 & \le \frac{1}{2}\left( {{t}^{\nu }}+\left( 1-\nu  \right)+\nu t \right) \\ 
 & \le \left( 1-\nu  \right)+\nu t.  
\end{aligned}
\end{equation}
Putting $t=\frac{b}{a}\neq 1$ in inequalities \eqref{sec2_eq12_02} and multiplying by $a$ to both sides, we deduce the sequence of inequalities.
\end{proof}

The following corollary gives an interpolation between the weighted geometric mean and the weighted logarithmic mean by the self-improving inequality technique.
\begin{corollary}\label{sec2_corollary2.9.}
Let $m\in\mathbb{N}$ and $0<\nu <1$. Then for any $t>0$, 
{\small
\begin{equation}\label{sec2_corollary2.9._ineq01}
\begin{aligned}
 & t^v\le \cdots \le t^{\left(1-\frac{1}{2^m}\right)\nu }f_\nu \left(t^{\frac{1}{2^m}}\right)\le t^{\left(1-\frac{1}{2^{m-1}}\right)\nu }f_\nu \left(t^{\frac{1}{2^{m-1}}}\right)\le\cdots\le t^{\left(1-\frac{1}{4}\right)\nu }f_\nu \left(t^{\frac{1}{4}}\right) \le t^{\frac{\nu }{2}}f_\nu \left(t^{\frac{1}{2}}\right)\le f_\nu (t),
\end{aligned}
\end{equation}
}
where the function $f_\nu (t)$ is defined as in \eqref{rep_func_log_mean}.
\end{corollary}
\begin{proof}
It is sufficient to prove the third inequality in \eqref{sec2_corollary2.9._ineq01} for any $m\in \mathbb{N}$. In the process of the proof of Theorem \ref{theorem_2.8.}, we found the inequality $t^{\frac{\nu }{2}}f_\nu \left(t^{\frac{1}{2}}\right)\le f_\nu (t)$ for $t>0$ and $0<\nu <1$. In this inequality, we put $t=s^{\frac{1}{2^{m-1}}}$. Then we have $s^{\frac{\nu }{2^m}}f_\nu \left(s^{\frac{1}{2^m}}\right)\le f_\nu \left(s^{\frac{1}{2^{m-1}}}\right)$. Multiplying $s^{\frac{2^{m-1}-1}{2^{m-1}}\nu }$ to both sides of this inequality, we get the third inequality in \eqref{sec2_corollary2.9._ineq01}.
Taking $m\to \infty$, we have $t^{\left(1-\frac{1}{2^m}\right)\nu }f_\nu \left(t^{\frac{1}{2^m}}\right)\to t^v$, since $f_v(1)=\lim\limits_{t\to1}f_v(t)=1$.
\end{proof}

Before expressing the next result, we recall an interesting inequality for convex functions \cite{drag}: If $f$ is a convex function on the interval $J\subseteq \mathbb{R}$, then for any $x,y\in J$,
\begin{equation}\label{2}
2r\left( f\left( x  \right)\nabla f\left( y \right)-f\left( x \nabla y \right) \right)
\le f\left( x  \right){{\nabla }_{t}}f\left( y \right)-f\left( x {{\nabla }_{t}}y \right)
\end{equation}
holds, where $r=\min \left\{ t,1-t \right\}$ and $0 \le t \le 1$. In the same paper, it has been shown that
\begin{equation}\label{5}
f\left( x \right){{\nabla }_{t}}f\left( y \right)-f\left( x {{\nabla }_{t}}y \right)\le 2R\left( f\left( x  \right)\nabla f\left( y \right)-f\left( x \nabla y \right) \right)
\end{equation}
where $R=\max \left\{ t,1-t \right\}$. 

The following theorem provides an improvement and a reverse for the first inequality in  \eqref{Pal_gen_HHI}, with the help of \eqref{2} and \eqref{5}.
\begin{theorem}\label{3}
Let $f : [a, b] \to \mathbb R$ be a convex function. Then for any $0 \le \nu  \le 1$,
\[2r\int\limits_{0}^{1}{\left( \left( f\left( a{{\nabla }_{\nu \lambda }}b \right)\nabla f\left( b{{\nabla }_{\left( 1-\nu  \right)\lambda }}a \right) \right)-f\left( a{{\nabla }_{\frac{1+\lambda \left( 2\nu -1 \right)}{2}}}b \right) \right)d\lambda }\le {{\mathfrak C }_{f,\nu }}\left( a,b \right)-f\left( a{{\nabla }_{\nu }}b \right),\]
and
\[{{\mathfrak C }_{f,\nu }}\left( a,b \right)-f\left( a{{\nabla }_{\nu }}b \right)\le 2R\int\limits_{0}^{1}{\left( \left( f\left( a{{\nabla }_{\nu \lambda }}b \right)\nabla f\left( b{{\nabla }_{\left( 1-\nu  \right)\lambda }}a \right) \right)-f\left( a{{\nabla }_{\frac{1+\lambda \left( 2\nu -1 \right)}{2}}}b \right) \right)d\lambda },\]
where $r=\min \left\{ \nu ,1-\nu  \right\}$ and $R=\max \left\{ \nu ,1-\nu  \right\}$.
\end{theorem}
\begin{proof}
By substituting $x =a\nabla_{\nu \lambda } b$  and $y=b\nabla_{(1-\nu )\lambda }a$, in \eqref{2}, we obtain
\[\begin{aligned}
   f\left( \left( a\nabla_{\nu \lambda } b \right){{\nabla }_{\nu }}\left( b\nabla_{(1-\nu )\lambda }a \right) \right)  &\le f\left( a\nabla_{\nu \lambda } b \right){{\nabla }_{\nu }}f\left( b\nabla_{(1-\nu )\lambda }a\right) \\ 
 &\quad -2r\left( f\left( a\nabla_{\nu \lambda } b \right)\nabla f\left(b\nabla_{(1-\nu )\lambda }a \right)
 -f\left( \left( a\nabla_{\nu \lambda } b\right)\nabla \left( b\nabla_{(1-\nu )\lambda }a \right) \right) \right).  
\end{aligned}\]
Since $\left(a\nabla_{\nu \lambda }b\right)\nabla_\nu \left(b\nabla_{(1-\nu )\lambda }a\right)=\left(a\nabla_{\nu \lambda }b\right)\nabla_\nu \left(a\nabla_{1-(1-\nu )\lambda }b\right)=a\nabla_{(1-\nu )\nu \lambda +\nu (1-(1-\nu )\lambda )}b=a\nabla_\nu b$, we have
\[f\left( a{{\nabla }_{\nu }}b \right)=f\left( \left( a\nabla_{\nu \lambda } b\right){{\nabla }_{\nu }}\left( b\nabla_{(1-\nu )\lambda }a \right) \right).\]
Consequently, we prove
\[\begin{aligned}
   f\left( a{{\nabla }_{\nu }}b \right)&\le f\left( a{{\nabla }_{\nu \lambda }}b \right){{\nabla }_{\nu}}f\left( b{{\nabla }_{\left( 1-\nu  \right)\lambda }}a \right) \\ 
 &\quad -2r\left( f\left( a{{\nabla }_{\nu \lambda }}b \right)\nabla f\left( b{{\nabla }_{\left( 1-\nu  \right)\lambda }}a \right)\left. -f\left( a{{\nabla }_{\frac{1+\left( 2\nu -1 \right)\lambda }{2}}}b \right) \right) \right).  
\end{aligned}\]
By taking integral over $\lambda  \in \left[ 0,1 \right]$, we reach to
\[\begin{aligned}
   f\left( a{{\nabla }_{\nu }}b \right)&\le \int\limits_{0}^{1}{\left( f\left( a{{\nabla }_{\nu \lambda }}b \right){{\nabla }_{\nu}}f\left( b{{\nabla }_{\left( 1-\nu  \right)\lambda }}a \right) \right)d\lambda } \\ 
 &\quad -2r\int\limits_{0}^{1}{\left( f\left( a{{\nabla }_{\nu \lambda }}b \right)\nabla f\left( b{{\nabla }_{\left( 1-\nu  \right)\lambda }}a \right)\left. -f\left( a{{\nabla }_{\frac{1+\left( 2\nu -1 \right)\lambda }{2}}}b \right) \right) \right)}d\lambda,  
\end{aligned}\]
which is the first inequality.
The second inequality follows likewise by employing inequality \eqref{5} instead of inequality \eqref{2}.
\end{proof}

\begin{remark}\label{remark_2.12_for_alpha}
Note that $$
f\left( a{{\nabla }_{\nu \lambda }}b \right)\nabla f\left( b{{\nabla }_{\left( 1-\nu  \right)\lambda }}a \right) -f\left( a{{\nabla }_{\frac{1+\lambda \left( 2\nu -1 \right)}{2}}}b \right)  \ge 0,\quad (0\le \nu,\,\lambda \le 1,\,\,a,b>0)
$$
by the convexity of $f$.
\end{remark}

The following corollary gives a refinement and a reverse for the inequality $a^{1-\nu }b^\nu \le L_\nu (a,b)$.
\begin{corollary}
Let $a,b>0$ and $0<\nu<1$ with $\nu \neq 1/2$. Then
\begin{equation}\label{cor_2.13_eq01}
\frac{r}{\log b-\log a}\left(\frac{a^{1-\nu }b^\nu -a}{\nu }+\frac{b-a^{1-\nu }b^\nu }{1-\nu }-\frac{4(\sqrt{ab}-a^{1-\nu }b^\nu )}{1-2\nu }\right)\le L_\nu (a,b)-a^{1-\nu }b^\nu, 
\end{equation}
and
\begin{equation}\label{cor_2.13_eq02}
L_\nu (a,b)-a^{1-\nu }b^\nu \le \frac{R}{\log b-\log a}\left(\frac{a^{1-\nu }b^\nu -a}{\nu }+\frac{b-a^{1-\nu }b^\nu }{1-\nu }-\frac{4(\sqrt{ab}-a^{1-\nu }b^\nu )}{1-2\nu }\right),
\end{equation}
where $r=\min \left\{ \nu ,1-\nu  \right\}$ and $R=\max \left\{ \nu ,1-\nu  \right\}$. In the limit of $\nu \to 1/2$, both sides in inequalities \eqref{cor_2.13_eq01} and \eqref{cor_2.13_eq02} coincide.
\end{corollary}
\begin{proof}
Letting $f\left( t \right)={{e}^{t}}$ in Theorem \ref{3}. A simple calculation reveals that
\[\int\limits_{0}^{1}{{{e}^{a{{\nabla }_{\nu \lambda }}b}}d\lambda }=\frac{{{e}^{a{{\nabla }_{\nu }}b}}-e^a}{\nu \left( b-a \right)},\quad\int\limits_{0}^{1}{{{e}^{b{{\nabla }_{\left( 1-\nu  \right)\lambda }}a}}d\lambda }=\frac{{{e}^{b}}-{{e}^{a{{\nabla }_{\nu }}b}}}{\left( 1-\nu  \right)\left( b-a \right)},\]
and
\[\int\limits_{0}^{1}{{{e}^{a{{\nabla }_{\frac{1+\lambda \left( 2\nu -1 \right)}{2}}}b}}d\lambda }=\frac{2\left({{e}^{a\nabla b}}-{{e}^{a{{\nabla }_{\nu }}b}}\right)}{\left( 1-2\nu  \right)\left( b-a \right)}.\]
Hence we have
{\small
\[r\left( \frac{{{e}^{a{{\nabla }_{\nu }}b}}-e^a}{\nu \left( b-a \right)}+\frac{{{e}^{b}}-{{e}^{a{{\nabla }_{\nu }}b}}}{\left( 1-\nu  \right)\left( b-a \right)}-\frac{4\left({{e}^{a\nabla b}}-{{e}^{a{{\nabla }_{\nu }}b}}\right)}{\left( 1-2\nu  \right)\left( b-a \right)} \right)\le \left( \frac{{{e}^{a{{\nabla }_{\nu }}b}}-e^a}{\nu \left( b-a \right)} \right){{\nabla }_{\nu }}\left( \frac{{{e}^{b}}-{{e}^{a{{\nabla }_{\nu }}b}}}{\left( 1-\nu  \right)\left( b-a \right)} \right)-{{e}^{a{{\nabla }_{\nu }}b}},\]
}
and
{\small
\[\left( \frac{{{e}^{a{{\nabla }_{\nu }}b}}-e^a}{\nu \left( b-a \right)} \right){{\nabla }_{\nu }}\left( \frac{{{e}^{b}}-{{e}^{a{{\nabla }_{\nu }}b}}}{\left( 1-\nu  \right)\left( b-a \right)} \right)-{{e}^{a{{\nabla }_{\nu }}b}}\le R\left( \frac{{{e}^{a{{\nabla }_{\nu }}b}}-e^a}{\nu \left( b-a \right)}+\frac{{{e}^{b}}-{{e}^{a{{\nabla }_{\nu }}b}}}{\left( 1-\nu  \right)\left( b-a \right)}-\frac{4\left({{e}^{a\nabla b}}-{{e}^{a{{\nabla }_{\nu }}b}}\right)}{\left( 1-2\nu  \right)\left( b-a \right)} \right).\]
}
We obtain desired inequalities by replacing $e^a$ and $e^b$ by $a$ and $b$ in the above two inequalities.

Finally, we quickly find that
$$
\lim_{\nu\to 1/2}\frac{4\left(\sqrt{ab}-a^{1-\nu}b^{\nu}\right)}{1-2\nu}
=\lim_{\nu\to 1/2}\frac{4a^{1-\nu}b^{\nu}\left(\log a-\log b\right)}{-2}
=2\sqrt{ab}\left(\log b-\log a\right),
$$
which implies the last statement by simple calculations.
\end{proof}

\begin{corollary}\label{4}
Let $f : [a, b] \to \mathbb R$ be a convex function. Then for any $0 \le \nu  \le 1$,
\[f\left( a{{\nabla }_{\nu }}b \right)\le f\left( a \right){{\nabla }_{\nu }}f\left( b \right)-2r\int\limits_{0}^{1}{\left( \left( f\left( a{{\nabla }_{\nu \lambda }}b \right)\nabla f\left( b{{\nabla }_{\left( 1-\nu  \right)\lambda }}a \right) \right)-f\left( a{{\nabla }_{\frac{1+\lambda \left( 2\nu -1 \right)}{2}}}b \right) \right)d\lambda },\]
where $r=\min \left\{ \nu ,1-\nu  \right\}$.
\end{corollary}
\begin{proof}
By the second inequality in \eqref{Pal_gen_HHI}, we understand that
\[{\mathfrak C _{f,\nu }}\left( a,b \right)\le f\left( a \right){{\nabla }_{\nu }}f\left( b \right).\]
Combining this with Theorem \ref{3} finishes the proof.
\end{proof}

The following result improves the second inequality in \eqref{7}.
\begin{corollary}
Let $f : [a, b] \to \mathbb R$ be a convex function. Then
{\small
\[\begin{aligned}
  & \int\limits_{0}^{1}{f\left( a{{\nabla }_{\nu }}b \right)d\nu } \\ 
 & \le f\left( a \right)\nabla f\left( b \right)-\int\limits_{0}^{1}{\left( \left( 1-\left| 2\nu -1 \right| \right)\int\limits_{0}^{1}{\left( \left( f\left( a{{\nabla }_{\nu \lambda }}b \right)\nabla f\left( b{{\nabla }_{\left( 1-\nu  \right)\lambda }}a \right) \right)-f\left( a{{\nabla }_{\frac{1+\lambda \left( 2\nu -1 \right)}{2}}}b \right) \right)d\lambda } \right)d\nu }. 
\end{aligned}\]
}
\end{corollary}
\begin{proof}
Since $2\min \left\{ \nu ,1-\nu  \right\}=1-\left| 2\nu -1 \right|$, we infer from Corollary \ref{4} that
\[\begin{aligned}
  & f\left( a{{\nabla }_{\nu }}b \right) \\ 
 & \le f\left( a \right){{\nabla }_{\nu }}f\left( b \right)-\left( 1-\left| 2\nu -1 \right| \right)\int\limits_{0}^{1}{\left( \left( f\left( a{{\nabla }_{\nu \lambda }}b \right)\nabla f\left( b{{\nabla }_{\left( 1-\nu  \right)\lambda }}a \right) \right)-f\left( a{{\nabla }_{\frac{1+\lambda \left( 2\nu -1 \right)}{2}}}b \right) \right)d\lambda }.  
\end{aligned}\]
We obtain the desired result if we take integral over $\nu \in \left[ 0,1 \right]$.
\end{proof}

\begin{remark}
The case $\nu ={1}/{2}\;$ in Corollary \ref{4}, recovers the second inequality of \eqref{7}.
Indeed,
\[\begin{aligned}
   f\left( a\nabla b \right)&\le f\left( a \right)\nabla f\left( b \right)-\left( \int\limits_{0}^{1}{\left( f\left( a{{\nabla }_{\frac{\lambda }{2}}}b \right)\nabla f\left( b{{\nabla }_{\frac{\lambda }{2}}}a \right)-f\left( a\nabla b \right) \right)d\lambda } \right) \\ 
 & =f\left( a \right)\nabla f\left( b \right)-\int\limits_{0}^{1}{\left( f\left( a{{\nabla }_{\frac{\lambda }{2}}}b \right)\nabla f\left( b{{\nabla }_{\frac{\lambda }{2}}}a \right) \right)d\lambda }+f\left( a\nabla b \right) \\ 
 & =f\left( a \right)\nabla f\left( b \right)-\int\limits_{0}^{1}{\left( f\left( a{{\nabla }_{\frac{\lambda }{2}}}b \right)\nabla f\left( a{{\nabla }_{1-\frac{\lambda }{2}}}b \right) \right)d\lambda }+f\left( a\nabla b \right).  
\end{aligned}\]
Equalities 
$\int_0^1f\left(a\nabla_{\frac{\lambda}{2}}b\right)d\lambda=2\int_0^{1/2}f\left(a\nabla_x b\right)dx$ and $\int_0^1f\left(a\nabla_{1-\frac{\lambda}{2}}b\right)d\lambda=2\int_{1/2}^1f\left(a\nabla_x b\right)dx$ imply
$$\int\limits_{0}^{1}{\left( f\left( a{{\nabla }_{\frac{\lambda }{2}}}b \right)\nabla f\left( a{{\nabla }_{1-\frac{\lambda }{2}}}b \right) \right)d\lambda }=\int\limits_0^1f\left(a\nabla_{\lambda}b\right)d\lambda.$$
Thus we have
$$
\int\limits_0^1f\left(a\nabla_{\lambda}b\right)d\lambda\le  f\left( a \right)\nabla f\left( b \right).
$$
\end{remark}

The following result gives a refinement of the second inequality in \eqref{Pal_gen_HHI}.
\begin{theorem}\label{11}
Let $f : [a, b] \to \mathbb R$ be a convex function. Then for any $0 \le \nu  \le 1$,
\[\begin{aligned}
   2\widetilde{r}(\nu)\left( f\left( a \right)\nabla f\left( b \right)-f\left( a\nabla b \right) \right)&\le f\left( a \right){{\nabla }_{\nu }}f\left( b \right)-{{\mathfrak C}_{f,\nu }}\left( a,b \right) \\ 
 & \le 2\widetilde{R}(\nu)\left( f\left( a \right)\nabla f\left( b \right)-f\left( a\nabla b \right) \right)  
\end{aligned}\]
where 
$$\widetilde{r}(\nu):=\int\limits_{0}^{1}{\left( {{r}_{1}}{{\nabla }_{\nu }}{{r}_{2}} \right)d\lambda },\quad \widetilde{R}(\nu):=\int\limits_{0}^{1}{\left( {{R}_{1}}{{\nabla }_{\nu }}{{R}_{2}} \right)d\lambda },$$
$r_1=\min \left\{ \nu\lambda,1-\nu\lambda \right\}$, $r_2=\min \left\{ (1-\nu)\lambda,1-(1-\nu)\lambda \right\}$, $R_1=\max \left\{ \nu\lambda,1-\nu\lambda \right\}$, and $R_2=\max \left\{ (1-\nu)\lambda,1-(1-\nu)\lambda \right\}$.
\end{theorem}
\begin{proof}
If we take $\alpha=f(a)$ and $\beta=f(b)$, in the equality $\left(\alpha\nabla_{\nu \lambda }\beta\right)\nabla_\nu \left(\beta\nabla_{(1-\nu )\lambda }\alpha\right)=\alpha\nabla_\nu \beta$, we deduce $\left(f(a)\nabla_{\nu \lambda }f(b)\right)\nabla_\nu \left(f(b)\nabla_{(1-\nu )\lambda }f(a)\right)=f(a)\nabla_\nu f(b)$, which implies the following:
\begin{equation}\label{12}
\begin{aligned}
   & f(a)\nabla_\nu f(b)-f\left( a{{\nabla }_{\nu \lambda }}b \right){{\nabla }_{\nu}}f\left( b{{\nabla }_{\left( 1-\nu  \right)\lambda }}a \right)\\
   &=\left(f(a)\nabla_{\nu \lambda }f(b)\right)\nabla_\nu \left(f(b)\nabla_{(1-\nu )\lambda }f(a)\right)-f\left( a{{\nabla }_{\nu \lambda }}b \right){{\nabla }_{\nu}}f\left( b{{\nabla }_{\left( 1-\nu  \right)\lambda }}a \right) \\ 
 &=\left(f(a)\nabla_{\nu \lambda }f(b)-f\left( a{{\nabla }_{\nu \lambda }}b \right)\right)\nabla_\nu \left(f(b)\nabla_{(1-\nu )\lambda }f(a)-f\left( b{{\nabla }_{\left( 1-\nu  \right)\lambda }}a \right)\right).  
\end{aligned}
\end{equation}
If we replace $t$ by $\nu\lambda$ in \eqref{2} and \eqref{5}, then we deduce
\begin{equation}\label{13}
2r_1\left( f\left( a  \right)\nabla f\left( b \right)-f\left( a \nabla b \right) \right)
\le f\left( a  \right){{\nabla }_{\nu\lambda}}f\left( b \right)-f\left( a {{\nabla }_{\nu\lambda}}b \right)\le 2R_1\left( f\left( a  \right)\nabla f\left( b \right)-f\left( a \nabla b \right) \right)
\end{equation}
where $r_1=\min \left\{ \nu\lambda,1-\nu\lambda \right\}$ and $R_1=\max \left\{ \nu\lambda,1-\nu\lambda \right\}$. In the same manner, if we replace $t$ by $(1-\nu)\lambda$ in \eqref{2} and \eqref{5}, then we obtain
\begin{equation}\label{14}
2r_2\left( f\left( a  \right)\nabla f\left( b \right)-f\left( a \nabla b \right) \right)
\le f\left( a  \right){{\nabla }_{(1-\nu)\lambda}}f\left( b \right)-f\left( a {{\nabla }_{(1-\nu)\lambda}}b \right)\le 2R_2\left( f\left( a  \right)\nabla f\left( b \right)-f\left( a \nabla b \right) \right)
\end{equation}
where $r_2=\min \left\{ (1-\nu)\lambda,1-(1-\nu)\lambda \right\}$ and $R_2=\max \left\{ (1-\nu)\lambda,1-(1-\nu)\lambda \right\}$.
Using equality \eqref{12} and inequalities \eqref{13} and \eqref{14}, we find the following inequality
\[\begin{aligned}
 & 2(r_1\nabla_{\nu}r_2)\left( f\left( a  \right)\nabla f\left( b \right)-f\left( a \nabla b \right) \right)\\
&\le \left(f(a)\nabla_{\nu \lambda }f(b)-f\left( a{{\nabla }_{\nu \lambda }}b \right)\right)\nabla_\nu \left(f(b)\nabla_{(1-\nu )\lambda }f(a)-f\left( b{{\nabla }_{\left( 1-\nu  \right)\lambda }}a \right)\right)\\
&\le 2(R_1\nabla_{\nu}R_2)\left( f\left( a  \right)\nabla f\left( b \right)-f\left( a \nabla b \right) \right).
\end{aligned}\]
Therefore, we obtain 
\[\begin{aligned}
& 2(r_1\nabla_{\nu}r_2)\left( f\left( a  \right)\nabla f\left( b \right)-f\left( a \nabla b \right) \right)\\
&\le f(a)\nabla_\nu f(b)-f\left( a{{\nabla }_{\nu \lambda }}b \right){{\nabla }_{\nu}}f\left( b{{\nabla }_{\left( 1-\nu  \right)\lambda }}a \right)\\
&\le 2(R_1\nabla_{\nu}R_2)\left( f\left( a  \right)\nabla f\left( b \right)-f\left( a \nabla b \right) \right).
\end{aligned}\]
By taking integral over $\lambda  \in \left[ 0,1 \right]$, we deduce the inequalities of the statement.
\end{proof}

Finding a maximum value of $\widetilde{r}(\nu)$ and a minimum value of $\widetilde{R}(\nu)$ for $0<\nu <1$, we state the following corollary.

\begin{corollary}\label{corollary_max_min}
Let $f : [a, b] \to \mathbb R$ be a convex function. Then for any $0 < \nu  < 1$,
$$
   \frac{1}{2}\left( f\left( a \right)\nabla f\left( b \right)-f\left( a\nabla b \right) \right)\le f\left( a \right){{\nabla }_{\nu }}f\left( b \right)-{{\mathfrak C}_{f,\nu }}\left( a,b \right) 
 \le \frac{3}{2}\left( f\left( a \right)\nabla f\left( b \right)-f\left( a\nabla b \right) \right).  
$$
\end{corollary}

\begin{proof}
To calculate the constants $\widetilde{r}(\nu)$ and $\widetilde{R}(\nu)$ appeared in Theorem \ref{11}, notice that
\[\int\limits_{0}^{1}{{{r}_{1}}d\lambda }=-\frac{\left( 2\nu -1 \right)\left| 2\nu -1 \right|-4\nu +1}{8\nu }\quad\text{ and }\quad\int\limits_{0}^{1}{{{r}_{2}}d\lambda }=-\frac{\left( 2\nu -1 \right)\left| 2\nu -1 \right|-4\nu +3}{8\left( \nu -1 \right)},\]
\[\int\limits_{0}^{1}{{{R}_{1}}d\lambda }=\frac{\left( 2\nu -1 \right)\left| 2\nu -1 \right|+4\nu +1}{8\nu }\quad\text{ and }\quad\int\limits_{0}^{1}{{{R}_{2}}d\lambda }=\frac{\left( 2\nu -1 \right)\left| 2\nu -1 \right|+4\nu -5}{8\left( \nu -1 \right)}.\]
Accordingly,
\[\widetilde{r}(\nu)=\frac{{{\left( 2\nu -1 \right)}^{2}}\left| 2\nu -1 \right|+6\nu \left( 1-\nu  \right)-1}{8\nu \left( 1-\nu  \right)},\]
and for $0< \nu < 1$
\[\widetilde{R}(\nu)=\frac{1+2\nu \left( 1-\nu  \right)-{{\left( 2\nu -1 \right)}^{2}}\left| 2\nu -1 \right|}{8\nu \left( 1-\nu  \right)}.\]
One can easily check that
	\[\frac{d\widetilde{r}\left( \nu  \right)}{d\nu}=-\frac{\left( 2\nu -1 \right)\left( \left( 2\nu \left( \nu -1 \right)-1 \right)\left| 2\nu -1 \right|+1 \right)}{8{{\nu }^{2}}{{\left( 1-\nu  \right)}^{2}}}\]
	and for $0< \nu < 1$
	$$
	\frac{\widetilde{r}\left(\nu \right)}{d\nu}=0 \Longleftrightarrow \nu=\frac12.
	$$
A direct computation shows that	
	\[\left\{ \begin{aligned}
  & \widetilde{r}'\left( \nu  \right)>0;\quad\text{ if }0<\nu < \frac{1}{2}, \\ 
 & \widetilde{r}'\left( \nu  \right)<0;\quad\text{ if }\frac{1}{2}< \nu <1. \\ 
\end{aligned} \right.\]
Notice that $\underset{0<\nu <1}{\mathop{\max }}\,\widetilde{r}\left( \nu  \right)={1}/{4}\;$, for $\nu ={1}/{2}\;$. Besides,
	\[\frac{\widetilde{R}\left( \nu  \right)}{d\nu}=\frac{\left( 2\nu -1 \right)\left( \left( 2\nu \left( \nu -1 \right)-1 \right)\left| 2\nu -1 \right|+1 \right)}{8{{\nu }^{2}}{{\left( 1-\nu  \right)}^{2}}}\]
	and
	$$
	\frac{\widetilde{R}\left( \nu  \right)}{d\nu}=0\Longleftrightarrow \nu=\frac12.
	$$
Direct calculations show that	
	\[\left\{ \begin{aligned}
  & \widetilde{R}'\left( \nu  \right)<0;\quad\text{ if }0<\nu < \frac{1}{2}, \\ 
 & \widetilde{R}'\left( \nu  \right)>0;\quad\text{ if }\frac{1}{2}< \nu <1. \\ 
\end{aligned} \right.\]
Notice that $\underset{0<\nu <1}{\mathop{\min }}\,\widetilde{R}\left( \nu  \right)={3}/{4}\;$, for $\nu ={1}/{2}\;$.
We thus have inequalities in this corollary.

\end{proof}

\begin{remark}
If we take $\nu ={1}/{2}\;$ in Corollary \ref{corollary_max_min}, then we have
$$
\frac{1}{2}\left( f\left( a  \right)\nabla f\left( b \right)-f\left( a \nabla b \right) \right)\le
f(a)\nabla f(b)-\int\limits_0^1f(a\nabla_{\lambda}b)d\lambda
\le \frac{3}{2}\left( f\left( a  \right)\nabla f\left( b \right)-f\left( a \nabla b \right) \right),
$$
which is equivalent to
$$
\frac{3}{2}f(a\nabla b)-\frac{1}{2}f(a)\nabla f(b)\le \int\limits_0^1f(a\nabla_{\lambda}b)d\lambda \le
\frac{1}{2}f(a\nabla b)+\frac{1}{2}f(a)\nabla f(b).
$$
The second inequality above represents Bullen's inequality (see, e.g., \cite{NP}, \cite{PPT}).
Naturally, the second inequality above improves the second inequality in \eqref{7}, since Theorem \ref{11} improves the second inequality of \eqref{Pal_gen_HHI} in a more general form.
\end{remark}

\begin{corollary}\label{corollary_identric}
Let $a,b>0$ and $0< \nu < 1$. Then
\begin{equation}\label{corollary_identric_eq01}
\widetilde{r}(\nu)\left( \sqrt{a}-\sqrt{b} \right)^2\le a{{\nabla }_{\nu }}b-{{L}_{\nu }}\left( a,b \right)\le \widetilde{R}(\nu)\left(\sqrt{a}-\sqrt{b}  \right)^2
\end{equation}
and
\begin{equation}\label{corollary_identric_eq02}
K{{\left( a,b \right)}^{\widetilde{r}(\nu)}}a{{\sharp}_{\nu }}b\le {{I}_{\nu }}\left( a,b \right)\le K{{\left( a,b \right)}^{\widetilde{R}(\nu)}}a{{\sharp}_{\nu }}b,
\end{equation}
where $\tilde{r}(\nu)$, $\tilde{R}(\nu)$, $r_1$, $r_2$, $R_1$, and $R_2$ are defined as in Theorem \ref{11}.
Also $K(a,b):=\dfrac{(a+b)^2}{4ab}$ is the Kantorovich constant and 
$$I_{\nu}(a,b):=\frac{1}{e}\left(a\nabla_\nu b\right)^{\frac{(1-2\nu)\left(a\nabla_{\nu}b\right)}{\nu(1-\nu)(b-a)}}\left(\frac{b^{\frac{\nu b}{1-\nu}}}{a^{\frac{(1-\nu)a}{\nu}}}\right)^{\frac{1}{b-a}}$$
 is the weighted identric mean introduced in \cite{PSMA2016}.
\end{corollary}
\begin{proof}
The results follow from in Theorem \ref{11}, by take $f(t)=e^t$ and $f(t)=-\log t$, respectively.
We note that ${\mathfrak C }_{e^x,\nu}(\log a,\log b)=L_{\nu}(a,b)$ and ${\mathfrak C }_{-\log x,\nu}(a,b)=\log\dfrac{1}{I_{\nu}(a,b)}$. The latter is due to
$$
\int\limits_{0}^{1}{f\left( b{{\nabla }_{\left( 1-\nu \right)(1-\lambda )}}a \right)d\lambda } =
 \int\limits_{0}^{1}{f\left( b{{\nabla }_{\left( 1-\nu \right) \mu}}a \right)d\mu }. 
$$
\end{proof}

Note that $\tilde{r}(\nu) \ge 0$ and $K(a,b)\ge 1$.
The first inequality and the second inequality in  \eqref{corollary_identric_eq01} respectively give a refinement and reverse of the inequality $L_{\nu}(a,b)\le a\nabla_{\nu} b$.
Also the first inequality and the second inequality in  \eqref{corollary_identric_eq02} respectively a refinement  and a reverse of the ineuality $a{{\sharp}_{\nu }}b\le {{I}_{\nu }}\left( a,b \right)$.

We reach the following result by combining Theorems \ref{3} and \ref{11}.
\begin{corollary}\label{0}
Let $f : [a, b] \to \mathbb R$ be a convex function. Then for any $0 \le \nu  \le 1$,
{\small
\[\begin{aligned}
  & 2\left( r\int\limits_{0}^{1}{\left( \left( f\left( a{{\nabla }_{\nu \lambda }}b \right)\nabla f\left( b{{\nabla }_{\left( 1-\nu  \right)\lambda }}a \right) \right)-f\left( a{{\nabla }_{\frac{1+\lambda \left( 2\nu -1 \right)}{2}}}b \right) \right)d\lambda }+\widetilde{r}(\nu)\left( f\left( a \right)\nabla f\left( b \right)-f\left( a\nabla b \right) \right) \right) \\ 
 & \le f\left( a \right){{\nabla }_{\nu }}f\left( b \right)-f\left( a{{\nabla }_{\nu }}b \right) \\ 
 & \le 2\left( R\int\limits_{0}^{1}{\left( \left( f\left( a{{\nabla }_{\nu \lambda }}b \right)\nabla f\left( b{{\nabla }_{\left( 1-\nu  \right)\lambda }}a \right) \right)-f\left( a{{\nabla }_{\frac{1+\lambda \left( 2\nu -1 \right)}{2}}}b \right) \right)d\lambda }+\widetilde{R}(\nu)\left( f\left( a \right)\nabla f\left( b \right)-f\left( a\nabla b \right) \right) \right)  
\end{aligned}\]
}
where $r=\min \left\{ \nu ,1-\nu  \right\}$, $R=\max \left\{ \nu ,1-\nu  \right\}$, and $\tilde{r}(\nu)$, $\tilde{R}(\nu)$ are defined as in Theorem \ref{11}.
\end{corollary}

\begin{remark}
Letting $f\left( t \right)=-\log t$ in Corollary \ref{0}. Simple calculations reveal that
\[\int\limits_{0}^{1}{-\log \left( a{{\nabla }_{\nu \lambda }}b \right)d\lambda }=1+\log {{\left( \frac{{{a}^{a}}}{{{\left( a{{\nabla }_{\nu }}b \right)}^{a{{\nabla }_{\nu }}b}}} \right)}^{\frac{1}{\nu \left( b-a \right)}}},\]
\[\int\limits_{0}^{1}{-\log \left( b{{\nabla }_{\left( 1-\nu  \right)\lambda }}a \right)d\lambda }=1+\log {{\left( \frac{{{\left( a{{\nabla }_{\nu }}b \right)}^{a{{\nabla }_{\nu }}b}}}{{{b}^{b}}} \right)}^{\frac{1}{\left( 1-\nu  \right)\left( b-a \right)}}},\]
and
\[\int_0^1 { - \log \left( {a{\nabla _{\frac{{1 + \left( {2\nu  - 1} \right)\lambda }}{2}}}b} \right)d\lambda }  = 1 + \log {\left( {\frac{{{{\left( {a\nabla b} \right)}^{a\nabla b}}}}{{{{\left( {a{\nabla _\nu }b} \right)}^{a{\nabla _\nu }b}}}}} \right)^{\frac{2}{{\left( {2v - 1} \right)\left( {b - a} \right)}}}}.\]
Thus, we have
\[\begin{aligned}
&\left( {\int\limits_0^1 { - \log \left( {a{\nabla _{\nu \lambda }}b} \right)d\lambda } } \right)\nabla \left( {\int\limits_0^1 { - \log \left( {b{\nabla _{\left( {1 - \nu } \right)\lambda }}a} \right)d\lambda } } \right) - \int\limits_0^1 { - \log \left( {a{\nabla _{\frac{{1 + \left( {2\nu  - 1} \right)\lambda }}{2}}}b} \right)d\lambda } \\
 &= \log {\left( {\frac{{{a^{\frac{a}{\nu }}}}}{{{b^{\frac{b}{{1 - \nu }}}}}}} \right)^{\frac{1}{{2\left( {b - a} \right)}}}}{\left( {\frac{{{{\left( {a{\nabla _\nu }b} \right)}^{\frac{{\left( {a{\nabla _\nu }b} \right)}}{{\nu \left( {1 - \nu } \right)\left( {2v - 1} \right)}}}}}}{{{{\left( {a\nabla b} \right)}^{\frac{{4\left( {a\nabla b} \right)}}{{\left( {2v - 1} \right)}}}}}}} \right)^{\frac{1}{{2\left( {b - a} \right)}}}}\ge 0.
\end{aligned}\]
The last inequality is due to Remark \ref{remark_2.12_for_alpha}.
Therefore, we obtain in terms of Kantorovich constant
\begin{equation}\label{remark2.22_eq01}
\alpha_{\nu}(a,b)^{r}{{ K(a,b) }^{\widetilde{r}(\nu)}}\le \frac{a{{\nabla }_{\nu }}b}{a{{\sharp}_{\nu }}b}\le \alpha_{\nu}(a,b)^{R}{{ K(a,b) }^{\widetilde{R}(\nu)}}
\end{equation}
where $\tilde{r}(\nu)$, $\tilde{R}(\nu)$ are defined as in Theorem \ref{11}, $r=\min \left\{ \nu ,1-\nu  \right\}$, $R=\max \left\{ \nu ,1-\nu  \right\}$, and
\begin{equation}\label{remark2.22_eq02}
\alpha_{\nu}(a,b) : = {\left( {\frac{{{a^{\frac{a}{\nu }}}}}{{{b^{\frac{b}{{1 - \nu }}}}}}\frac{{{{\left( {a{\nabla _\nu }b} \right)}^{\frac{{\left( {a{\nabla _\nu }b} \right)}}{{\nu \left( {1 - \nu } \right)\left( {2v - 1} \right)}}}}}}{{{{\left( {a\nabla b} \right)}^{\frac{{4\left( {a\nabla b} \right)}}{{\left( {2v - 1} \right)}}}}}}} \right)^{\frac{1}{{ b - a}}}}\ge 1,
\end{equation}
where $\nu\in(0,1/2)\cup(1/2,1)$. We easily find that $\lim\limits_{b\to a}\alpha_{\nu}(a,b)=1$ and
\begin{equation}\label{remark2.22_eq03}
\lim\limits_{\nu\to \frac{1}{2}}\alpha_{\nu}(a,b) = \left(e\left(a\nabla b\right)\left(\dfrac{a^a}{b^b}\right)^{\frac{1}{b-a}}\right)^2=\left(\frac{a\nabla b}{I_{1/2}(a,b)}\right)^2\ge 1.
\end{equation}
 Thus, the first and the second inequalities of \eqref{remark2.22_eq01}, respectively, give a refinement and a reverse of the weighted arithmetic-geometric mean inequality.
 
 In addition, \eqref{remark2.22_eq02} together with \eqref{remark2.22_eq03} give an upper bound for the weighted identric mean:
\begin{equation}\label{remark2.22_eq04}
 I_{\nu}(a,b)\le \frac{1}{e}\left(\frac{a^a}{b^b}\frac{\left(a\nabla b\right)^{\frac{4\left(a\nabla b\right)}{1-2\nu}}}{\left(a\nabla_{\nu}b\right)^{\frac{4\left(a\nabla_{\nu}b\right)}{1-2\nu}}}\right)^{\frac{1}{b-a}},\quad \left(0<\nu <1,\quad\nu\neq\frac12\right)
\end{equation}
 since 
 $$
 1\le \alpha_{\nu}(a,b) =\frac{1}{I_{\nu}(a,b)}\frac{1}{e}\left(\frac{a^a}{b^b}\frac{\left(a\nabla b\right)^{\frac{4\left(a\nabla b\right)}{1-2\nu}}}{\left(a\nabla_{\nu}b\right)^{\frac{4\left(a\nabla_{\nu}b\right)}{1-2\nu}}}\right)^{\frac{1}{b-a}}.
 $$
 Taking the limit of $\nu\to 1/2$ in \eqref{remark2.22_eq04}, we have the bound of $I_{1/2}(a,b)$ in the following way.
 $$
 I_{1/2}(a,b)\le e(a\nabla b)^2\left(\frac{a^a}{b^b}\right)^{\frac{1}{b-a}}=\frac{(a\nabla b)^2}{I_{1/2}(a,b)}
 $$ 
 which implies $I_{1/2}(a,b)\le a\nabla b$. This fits the inequality in \eqref{remark2.22_eq03}.
\end{remark}

\subsection*{Declarations}
\begin{itemize}
\item {\bf{Availability of data and materials}}: Not applicable.
\item {\bf{Competing interests}}: The authors declare that they have no competing interests.
\item {\bf{Funding}}: This research is supported by a grant (JSPS KAKENHI, Grant Number: 21K03341) awarded to the author, S. Furuichi.
\item {\bf{Authors' contributions}}: Authors declare that they have contributed equally to this paper. All authors have read and approved this version.
\end{itemize}


\begin{thebibliography}{99}
\bibitem{m1}
H. Alzer, C. M. da Fonseca, and A. Kova\v cec,  {\it Young-type inequalities and their matrix analogues}, Linear Multilinear Algebra., {\bf 63}(3) (2015), 622--635.

\bibitem{m2}
D. I. Cartwright, M. J. Field, {\it A refinement of the arithmetic mean-geometric mean inequality}, Proc. Am. Math. Soc., {\bf 71} (1978), 36--38.

\bibitem{drag}
S. S. Dragomir, {\it Bounds for the normalised Jensen functional}, Bull. Austral. Math. Soc., {\bf 3} (2006), 471--478.



\bibitem{DraAg}
S. S. Dragomir, R. P. Agarwal, {\it Two inequalities for differentiable mappings and applications to special means of real numbers and to trapezoidal formula}, Appl. Math. Lett., {\bf 11}(5) (1998), 91--95. 



\bibitem{FaLaBa} 
A. El Farissi, Z. Latreuch, and B. Balaidi, {\it Hadamard type inequalities for near convex functions}, Gazeta Matematica Seria A., {\bf 28}(157) (2010), 9--14.

\bibitem{F_Y_M} 
S. Furuichi, K. Yanagi, and H. R. Moradi, {\it  Mathematical inequalities on some weighted means}, J. Math. Inequal., to appear. \url{https://doi.org/10.48550/arXiv.2110.06493}

\bibitem{F_M} 
S. Furuichi, N. Minculete, {\it  Refined inequalities on the weighted inequalities}, J. Math. Inequal., {\bf 14}(4) (2020), 1347--1357. 

\bibitem{F_M_2} 
S. Furuichi, N. Minculete, {\it Bounds for the differences between arithmetic and geometric means and their applications to inequalities}, Symmetry., {\bf 13}(12) (2021). 
\url{https://doi.org/10.3390/sym13122398}

\bibitem{FM2020} 
S. Furuichi, H. R. Moradi, {\it  Advances in Mathematical Inequalities}, De Gruyter, 2020.

\bibitem{fms2021}
S. Furuichi, H. R. Moradi, and M. Sababheh, {\it New inequalities for interpolational operator means}, J. Math. Inequal., {\bf 15}(1) (2021), 107--116.


\bibitem{m3}
F. Kittaneh, Y. Manasrah, {\it Improved Young and Heinz inequalities for matrices}, J. Math. Anal. Appl., {\bf 361} (2010), 262--269.

\bibitem{m4}
F. Kittaneh, Y. Manasrah, {\it Reverse Young and Heinz inequalities for matrices}, Linear Multilinear Algebra., {\bf 59}  (2011), 1031--1037.



\bibitem{moradi1}
H. R. Moradi, M. Sababheh, and S. Furuichi, {\it On the operator Hermite-Hadamard inequality}. Complex Anal. Oper. Theory., {\bf 15}, 122 (2021). 
\url{https://doi.org/10.1007/s11785-021-01172-w}

\bibitem{NP} 
C. P. Niculescu, L.-E. Persson, {\it Convex Functions and Their Applications: A Contemporary Approach}, Springer-Verlag, New York, 2006.


\bibitem{PSMA2016} 
R. Pal, M. Singh, M. S. Moslehian, and J. S. Aujla, {\it A new class of operator monotone functions via operator means}, Linear Multilinear Algebra., {\bf 64}(12) (2016), 2463--2473.

\bibitem{PPT}
 J. Pe\v cari\' c, F. Proschan, and Y. L. Tong, {\it Convex Functions, Partial Orderings and Statistical Applications}, Academic Press, Inc., 1992.

\bibitem{moradi2}
M. Sababheh, H. R. Moradi, and S. Furuichi, {\it Integrals refining convex inequalities}, Bull. Malays. Math. Sci. Soc., {\bf 43} (2020), 2817--2833.


\end{thebibliography}
\end{document}